\documentclass[a4paper,reqno,11pt]{amsart}
\usepackage{t1enc}
\usepackage[margin=1.0in]{geometry}
\usepackage{ifthen}
\usepackage{graphicx}
\usepackage{amsmath,amstext,amssymb,amsopn,amsthm,color}

\newcommand{\del}[1]{\delta_1\left(#1\right)}
\newcommand{\deldB}[1]{\delta_B^2(#1)}
\newcommand{\delB}[1]{\delta_B(#1)}
\newcommand{\Hh}{H_0}

\newcommand{\kernel}{k}
\newcommand{\kk}{k_1}
\newcommand{\kt}{k(t,x,y)}
\newcommand{\kkt}{k_1(t,x,y)}

\newcommand{\anxy}{\alpha_{xy}}

\newcommand{\R}{\mathbf{R}}

\newcommand{\pr}{\mathbf{P}}
\newcommand{\ex}{\mathbf{E}}

\theoremstyle{plain}
\newtheorem{theorem}{Theorem}
\newtheorem{lemma}{Lemma}
\newtheorem{corollary}{Corollary}
\newtheorem{proposition}{Proposition}

\theoremstyle{definition}

\newtheorem{remark}{Remark}

\theoremstyle{remark}

\newcommand{\formula}[2][nolabel]
{\ifthenelse{\equal{#1}{nolabel}}
 {\begin{align*} #2 \end{align*}}
 {\ifthenelse{\equal{#1}{}}
  {\begin{align} #2 \end{align}}
  {\begin{align} \label{#1} #2 \end{align}}
 }
}

\numberwithin{equation}{section}

\begin{document}

%
%

\title[Dirichlet heat kernel for the Laplacian in a ball]{Dirichlet heat kernel for the Laplacian in a ball}
\author{Jacek Ma{\l}ecki, Grzegorz Serafin}
\address{Jacek Ma{\l}ecki, Grzegorz Serafin \\ Faculty of Pure and Applied Mathematics \\  Wroc{\l}aw University of Science and Technology \\ ul.
Wybrze{\.z}e Wyspia{\'n}\-skiego 27 \\ 50-370 Wroc{\l}aw,
Poland
}
\email{jacek.malecki@pwr.edu.pl, grzegorz.serafin@pwr.edu.pl}


\thanks{Jacek Ma\l{}ecki was supported by the National Science Centre grant no. 2013/11/D/ST1/02622. Grzegorz Serafin was financed
by the National Science Centre grant no. 015/18/E/ST1/00239.}

\begin{abstract} 
  We provide sharp two-sided estimates on the Dirichlet heat kernel $\kkt$ for the Laplacian in a ball. The result accurately describes the exponential behaviour of the kernel for small times and significantly improves the qualitatively sharp results known so far. As a consequence we obtain the full description of the kernel $\kkt$ in terms of its global two-sided sharp estimates.
\end{abstract}

\maketitle

\section{Introduction}
\label{sec:introduction}
Let $n\geq 1$ and denote by $\kkt$ the heat kernel of the Dirichlet Laplacian in the unit ball $B(0,1)=\{x\in \R^n: |x|<1\}$. The main result of the paper is the following theorem providing sharp two-sided estimates of $\kkt$ for the whole range of the space parameters $x,y\in B(0,1)$ and small times $t$.

\begin{theorem}
\label{thm:main}
For every $n\geq 1$ and $T>0$ there exists constant $C=C(n,T)>1$  such that 
\begin{eqnarray}
\label{eq:main}
\frac{1}{C}\,  \frac{h(t,x,y)}{t^{n/2}}\exp\left({-\frac{|x-y|^2}{4t}}\right)\leq \kkt\leq C\, \frac{h(t,x,y)}{t^{n/2}}\exp\left({-\frac{|x-y|^2}{4t}}\right)
\end{eqnarray}
for every $|x|,|y|<1$ and $t<T$, where 
\begin{align}
\label{eq:htxy:est}
  h(&t,x,y) =\left(1\wedge\frac{(1-|x|)(1-|y|)}t\right)+\left(1\wedge\frac{(1-|x|)|x-y|^2}t\right)\left(1\wedge\frac{(1-|y|)|x-y|^2}t\right)\/.
\end{align}
\end{theorem}
Due to the translation invariance and the scaling property of the Laplacian in $\R^n$, one can immediately obtain the corresponding result for the ball $B(x_0,r)$ with a radius $r>0$ and a center at $x_0\in\R^d$. The long-time behaviour (i.e. for $t\geq T$, where $T>0$ is fixed) of $\kkt$ can be easily deduce from the general theory (see \cite{Davies:1987}, \cite{DaviesSimon:1984}), i.e. there is a comparability between $\kkt$ and 
\begin{eqnarray*}
 (1-|x|)(1-|y|)e^{-\lambda_1 t}\/,
\end{eqnarray*}
for every $|x|,|y|<1$ and $t\geq T$, where $\lambda_1$ stands for the first eigenvalue of $-\Delta$ on $B(0,1)$. Note that this kind of result can be derived from the spectral series representation of the kernel $\kkt$ in terms of the eigenfunctions and eigenvalues of the Laplacian in the ball (see for example \cite{H}), i.e. it can be shown that for large times $t$ the first component of the series dominates the others. However, this representation is ineffective for small $t$, when we have to deal with the cancellations of highly oscillating series. Combining the long time behaviour result stated above together with Theorem \ref{thm:main} we easily obtain the global sharp two-sided estimates.

\begin{corollary}
\label{cor:globalestimates}
For every $n\geq 1$, there exists constant $C_1=C_1(n)>1$ such that 
\begin{align*}
   \frac{1}{C_1}\frac{h(t\wedge 1, x,y)}{(t\wedge 1)^{n/2}}\exp\left(-\frac{|x-y|^2}{4t}-\lambda_1 t\right)&\leq \kkt\leq C_1\frac{h(t\wedge 1, x,y)}{(t\wedge 1)^{n/2}}\exp\left(-\dfrac{|x-y|^2}{4t}-\lambda_1 t\right)
\end{align*}
for every $x,y\in B(0,1)$ and $t>0$\/.
\end{corollary}

Studies on the behaviour of heat kernels related to various kinds of operators and domains or manifolds have very long history and there is an enormous number of research papers on this topic including many beautiful and general results (see, among others, \cite{Berg:1990}, \cite{CheegerYau:1981}, \cite{CMM:2000}, \cite{GrigorSaloff:2002}, \cite{KimSong:2006}, \cite{Zhang:2003} and the references therein). On the other hand, it is difficult to imagine more classical example than the Laplace operator in smooth bounded domain and a unit ball is definitely the most basic example of such set. Nevertheless, such accurate result as in Theorem \ref{thm:main} has not been known until now except the one dimensional case. More precisely, for $n=1$ we have $B(0,1)=(-1,1)$ and, apart from  the usual spectral representation, the representation in terms of the series of differences of exponents is available. It leads to the estimates (see for example \cite{PSZ}) 
\begin{eqnarray}
   \label{eq:estimates:d1}
	  \kkt \approx \left(1\wedge \frac{(x+1)(y+1)}{t}\right)\left(1\wedge \frac{(1-x)(1-y)}{t}\right)\frac{1}{\sqrt{t}}\exp\left(-\frac{(x-y)^2}{4t}\right)\/,
\end{eqnarray}
for every $x,y\in (-1,1)$ and $t$ small enough. Although it requires some effort, it can be shown that the product of the two minimums above is comparable to $h(t,x,y)$. We emphasize that the one-dimensional case is significantly different from the multidimensional case $n\geq 2$, since for $n=1$ ``being close to the boundary'' just means ``being close to $-1$ or $1$'', which makes the consideration much simpler.

To outline the context of Theorem \ref{thm:main} we recall the upper-bounds for $\kkt$ provided by E.~B.~Davis in \cite{Davies:1987}. The result relates to much more general setting of bounded $C^{1,1}$ domains, but limited only to the case of a unit centered ball it ensures existence of constants $c_1,c_2>0$  and $T>0$ such that
\begin{eqnarray*}
  \label{eq:Davies:est}
   \kkt \leq \left[\frac{(1-|x|)(1-|y|)}{t}\wedge 1\right]\frac{c_1}{t^{d/2}}\exp\left({-c_2\frac{|x-y|^2}{t}}\right)
\end{eqnarray*}
for every $x,y\in B(0,1)$ and $t<T$. These bounds were complemented by Q.~S.~Zhang in \cite{Zhang:2002}, who proved that for some $c_3,c_4>0$
\begin{eqnarray*}
\label{eq:Zhang:est}
   \kkt \geq\left[\frac{(1-|x|)(1-|y|)}{t}\wedge 1\right]\frac{c_3}{t^{d/2}}\exp\left({-c_4\frac{|x-y|^2}{t}}\right)
\end{eqnarray*}
for every $x,y\in B(0,1)$ and $t$ small enough. As a consequence, we obtain quantitatively sharp estimates of $\kkt$. The obvious difference between the Davies-Zhang's result and the estimates given in Theorem \ref{thm:main} is that the latter accurately describes the exponential behaviour of $\kkt$, i.e. there are no different constants in the exponential factors in the lower and upper bounds. According to  \cite{Berg:1990} we expect that the exponential behaviour of $\kkt$ for small $t$ should be the same as in the case of the Gaussian kernel 
\begin{eqnarray}
   \label{eq:Gauss}
  \kt = \frac{1}{(4\pi t)^{n/2}}\exp\left(-\frac{|x-y|^2}{4t}\right)\/,\quad t>0\/,\quad x,y\in\R^n\/.
\end{eqnarray}
However, as Theorem \ref{thm:main} shows, it is not possible to get $c_2=c_4=1/4$ in the Davies-Zhang estimates, since the sharp estimates require modification of the non-exponential terms and appearance of the factor $h(t,x,y)$ described above. The form of the  factor is new and has not appeared, up to the  best
knowledge of the authors, in the literature so far.

The heat kernel $\kkt$ has the very well-known probabilistic interpretation as the transition probability density of the $n$-dimensional Brownian motion  killed when reaching the boundary of the ball. Thus, it can be expressed in terms of the Gauss kernel $\kt$ and the distribution of the first hitting time and hitting place of the sphere by the Brownian motion, i.e. the Hunt formula holds (see \eqref{eq:Hunt:general} below). On the other hand, the density of the joint distribution of the first hitting time and hitting place is a normal derivative of $\kkt$ (see \eqref{eq:hx:diff}) and consequently Theorem \ref{thm:main} immediately leads to its sharp two-sided estimates (see Corollary \ref{cor:hxtz:estimates}).  

The main advantage of the Hunt formula compared to the series representation is the simple fact that we represent the heat kernel as a difference of two non-negative expressions, which is much simpler to deal with than with the series of oscillating components. This approach has been successfully used in \cite{MSZ:2016} to study the short time behaviour of the Fourier-Bessel heat kernel. Since the Hunt formula is the starting point, we use several probabilistic tools and ideas in the proof of the main result. However, some parts of the proof are purely analytical. In both approaches we try to use as much geometric arguments as possible to make the proof simpler and applicable in other contexts and potential extensions.

Finally, the result stated in Theorem \ref{thm:main} should be discussed in the context of the famous Mark Kac's principle of not feeling the boundary stated in \cite{Kac:1951}. Restricting the result to the case of the ball, Kac showed that $k_1(t,x,y)\sim 1/(4\pi t)\exp(-|x-y|^2/(4t))$ (in $\R^2$) as $t\to 0$, where $x,y$ are fixed, i.e. the behaviour of $k_1(t,x,y)$ and $k(t,x,y)$ are the same in this sense, when $t$ goes to zero. He described this phenomenon in his famous paper \cite{Kac:1966} by saying 
\begin{center}
   \emph{As the Brownian particles begin to diffuse they are not aware, so to speak,\\ of
the disaster that awaits them when they reach the boundary}
\end{center}
Following this poetic language, we can now say that the Brownian particles do not have death premonitions when they begin to diffuse (the exponential behaviours of $k_1$ and $k$ are the same), but they are afraid of death by rational judgement of the distances to the threat of the starting and the final points, the length of the road between them and the time in which they should overcome this path (described in details by $h(t,x,y)$).

\section{Preliminaries}
\subsection{Notation}
We write $f\approx g$ whenever there exists constant $c>1$ depending only on a dimension $n$ such that $c^{-1}\leq f/g\leq c$ holds for the indicated range of the arguments of functions $f$ and $g$. Similarly, we write $f\lesssim g$ (or $f\gtrsim g $) if we have $f\leq c g$ ($f\geq c g$) for some constant $c>0$ depending only on $n$. If the constants appearing in the estimates depend on some other parameters, it will be indicated by placing those parameters above the sings $\approx$, $\lesssim$ and $\gtrsim$.

By $|x|$ we denote the Euclidean norm of a point $x\in \R^n$ and write $\anxy$ for the (smaller) angle between non-zero vectors $x$ and $y$. Moreover, we put $\anxy=0$ if $x$ or $y$ is zero. We write $B(x_0,r)=\{x\in\R^n:|x-x_0|<r\}$ for a ball of a radius $r>0$ centered at $x_0\in\R^n$ and $S(x_0,r)=\{x\in\R^n:|x-x_0|=r\}$ stands for the corresponding sphere. In the basic case $x_0=0$ and $r=1$ we write $d\sigma(z)$, $z\in S(0,1)$, for the spherical measure. For $x\in B(0,1)$, $x\neq 0$, we indicate by $H_x$ the half-space containing the unit ball $B(0,1)$ and such that its boundary hyperplane is tangent to the sphere $S(0,1)$ at the point $x/|x|$. In the special case $x/|x|=(1,0\ldots,0)$ we omit the subscript in the notation and we simply write $H=\{x\in\R^n:x_1<1\}$. For a general hyperplane $L$, we denote by $P_L(x)$ the reflection of $x$ with respect to $L$. In particular, we have
\begin{eqnarray}
\label{eq:reflection:plane}
P_{\partial H}(x)=(2-x_1,x_2,\ldots,x_n)\/,\quad x=(x_1,\ldots,x_n)\in\R^n\/.
\end{eqnarray}
Moreover, we put
\begin{eqnarray*}
\bar x=\frac{2-|x|}{|x|}\,x\/,
\end{eqnarray*}
whenever $x\neq 0$. If $x\in B(0,1)$, then $\bar{x}$ is a reflection of the point $x$ with respect to the hyperplane tangent to $S(0,1)$ at a point $x/|x|$, i.e. $\bar{x} = P_{\partial H_x}(x)$.

For a general set $D\subset \R^n$ and $x\in D$ we write $\delta_D(x)$ for a distance of $x$ to the boundary $\partial D$. As previously, we shorten the notation in the case of $D=B(0,1)$ and just write $\del{x} = \delta_{B(0,1)}(x)=1-|x|$. For every $x,y\in B(0,1)$, by the parallelogram law, we have
\begin{eqnarray*}
  \label{eq:parallel:basic}
  1-\left|\frac{x+y}{2}\right|^2 = \frac{|x-y|^2}{4}+\frac{1-|x|^2}{2}+\frac{1-|y|^2}{2}
\end{eqnarray*}
and consequently, since $ 2\left(1-\left|\frac{x+y}{2}\right|\right)\geq 1-\left|\frac{x+y}{2}\right|^2$, we obtain
\begin{eqnarray}
  \label{eq:parallel:low}
  \delta_1\left(\frac{x+y}{2}\right)\geq \frac{|x-y|^2}{8}+\frac{1-|x|}{4}+\frac{1-|y|}{4} \/.
\end{eqnarray}

\subsection{Brownian motion}
We consider $n$-dimensional Brownian motion $W=(W_t)_{t\geq 0}$ starting from $x\in\R^n$ and we denote by $\pr^x$ and $\ex^x$ the corresponding probability law and the expected value. Obviously $\pr^x$ is absolutely continuous with respect to the Lebesgue measure and $\kt$ is the corresponding transition probability density. 

The next lemma will be frequently used in the sequel. One can interpret the result, in the probabilistic context, by saying that Brownian motion going from $x$ to $y$ in time $2t$ is mostly at time $t$ passing through a neighbourhood of the midpoint $(x+y)/2$ of a size comparable to $\sqrt{t}$. In fact, we can move away from $(x+y)/2$ at a distance not greater then multiplicity of $\sqrt{t}$. 
\begin{lemma}
  \label{lemma:kt:lower}
  For every $c,l>0$ we have
	\begin{eqnarray*}
	   \int_{B(a,c\sqrt{t})}\kernel(t,x,z)\kernel(t,z,y)\,dz \stackrel{c,l}{\approx} \kernel(2t,x,y)
	\end{eqnarray*}
	for every $x,y\in \R^n$, $t>0$ and $a\in \R^n$ such that $|a-\frac{x+y}{2}|\leq l\sqrt{t}$.
\end{lemma}
\begin{proof}
The upper estimates are obvious and come directly from the Chapman-Kolmogorov identity. Thus, we focus only on the lower bounds. 
Without loss of generality we can and we do assume that $x=(-|x-y|/2,0,...,0)$,  $y=(|x-y|/2,0,...,0)$. Then $(x+y)/2=0$. Let $c$ and $l$ be fixed positive constants. For every $x,y\in \R^n$ of the form indicated above, $t>0$ and $a=(a_1,\ldots,a_n)\in \R^n$ we can write
\begin{eqnarray*}
   \int_{B(a,c\sqrt{t})}\kernel(t,x,z)\kernel(t,z,y)dz \geq \int_{a_1-c\sqrt{t/n}}^{a_1+c\sqrt {t/n}}...\int_{a_n-c\sqrt  {t/n}}^{a_n+c\sqrt {t/n}}\kernel(t,x,z)\kernel(t,z,y)dz_1...dz_n
\end{eqnarray*}
and the inequality follows since the ball $B(a,c\sqrt{t})$ contains the cuboid 
\begin{eqnarray*}
   (a_1-c\sqrt{t/n}, a_1+c\sqrt{t/n})\times \ldots\times (a_n-c\sqrt{t/n}, a_n+c\sqrt{t/n})\/.
\end{eqnarray*}
Due to the special form of $x$ and $y$, we can easily show that 
\begin{eqnarray*}
  |x-z|^2+|y-z|^2 &=& \frac{|x-y|^2}{2} + 2|z|^2
\end{eqnarray*}
and consequently
\begin{eqnarray*}
  \kernel(t,x,z)\kernel(t,z,y)= \frac{1}{(4\pi t)^{n}}\exp\left(-\frac{|x-y|^2}{8t}\right)\exp\left(-\frac{|z|^2}{2t}\right)\/.
\end{eqnarray*}
Combining all together we obtain the lower bound of the form
\begin{eqnarray*}
\frac{1}{(4\pi t)^{n}}\exp\left(-\frac{|x-y|^2}{8t}\right)\int_{a_1-c\sqrt{t/n}}^{a_1+c\sqrt {t/n}}...\int_{a_n-c\sqrt  {t/n}}^{a_n+c\sqrt {t/n}}e^{-\frac{|z|^2}{2t}}dz_1\ldots dz_n\/.
\end{eqnarray*}
Finally, assuming that $|a-(x+y)/2|=|a|\leq l\sqrt{t}$, which implies $|a_i|\leq l\sqrt{t}$, we get
\begin{eqnarray*}
   \prod_{i=1}^n \left(\int_{a_i-c\sqrt{t/n}}^{a_i+c\sqrt{t/n}}e^{-\frac{z_i^2}{2t}}dz_i\right) &=& t^{n/2}
  \prod_{i=1}^n	\left(\int_{a_i/\sqrt{t}-c/\sqrt{n}}^{a_i/\sqrt{t}+c/\sqrt{n}}e^{-\frac{u^2}{2}}du\right)
	> t^{n/2}\left(\frac{2c}{\sqrt{n}}e^{-\frac{(l+c/\sqrt{n})^2}{2}}\right)^n
	\end{eqnarray*}
	and we arrive at
	\begin{eqnarray*}
	 \int_{B(a,c\sqrt{t})}  \kernel(t,x,z)\kernel(t,z,y)\,dz \stackrel{n,c,l}{\gtrsim} \, \kernel(2t,x,y)\/.
	\end{eqnarray*}

\end{proof}

For a general smooth domain $D\subset \R^d$ we define the first exit time from $D$ by
\begin{eqnarray*}
   \tau_D = \inf\{t>0: W_t\notin D\}\/.
\end{eqnarray*}
We write $\kernel_D(t,x,y)$ for the transition probability density for $W^{D}=(W^D_{t})_{t\geq 0}$ Brownian motion killed upon leaving a set $D$. To shorten the notation we write $\tau_1$ (and obviously $\kkt$) in the case $D=B(0,1)$. The relation between $\kernel_D(t,x,y)$ and $\kt$ together with the joint distribution of $(\tau_D, W_{\tau_D})$ is described by the Hunt formula
\begin{eqnarray*}
   \label{eq:Hunt:general}
	\kernel_D(t,x,y) = \kt-\ex^x[t>\tau_D, W_{\tau_D},y]\/,\quad x,y\in D\/,\quad t>0\/.
\end{eqnarray*}
Due to the reflection principle, the case of a half-space is quite special. More precisely, for $H=\{x\in \R^n: x_1<1\}$  we can write $k_H(t,x,y)$ explicitly as follows
\begin{eqnarray*}
   \kernel_H(t,x,y) = \kt-\kernel(t,x,P_{\partial H}(y))\/,
\end{eqnarray*}
where $P_{\partial H}(y)=(2-y_1,y_2,\ldots,y_n)$ as defined in \eqref{eq:reflection:plane}. Since
\begin{eqnarray*}
   \kernel(t,x,P_{\partial H}(y)) &=& \exp\left(-\frac{(1-y_1)(1-x_1)}{t}\right)\kt\/,\quad x,y\in\R^n
\end{eqnarray*}
and $\delta_H(x)=1-x_1$, $\delta_H(y)=1-y_1$ we immediately get
\begin{eqnarray}	
	\label{eq:H:estimates}
	\kernel_H(t,x,y) = \kt-\kernel(t,x,P_{\partial H}(y)) &\approx& \left(1\wedge\frac{\delta_H(x)\delta_H(y)}{t}\right)\kt\/,x,y\in H\/. 
\end{eqnarray}
The last estimates hold for every half-space, since both sides are rotationally and translationally invariant. Such transparent formula and estimates are no longer available in the considered case of a unit ball and we have to start from the general formula
\begin{eqnarray}
\label{eq:Hunt:ball}
   \kkt = \kt-\int_0^t \int_{|z|=1}\kernel(t-s,z,y)q_z(s,z)dsd\sigma(z)\/,
\end{eqnarray}
where $q_x(t,z)$ denotes the density function of the joint distribution $(\tau_1,W_{\tau_1})$ for the process starting from $x\in B(0,1)$. Note also that we can recover $q_x(t,z)$ from $\kkt$ by differentiating in the norm direction (see \cite{H})
\begin{eqnarray}
 \label{eq:hx:diff}
  q_x(t,z) = \dfrac{\partial}{\partial n_z}\kernel_1(t,x,z)\/,\quad |x|<1\/,|z|=1\/,t>0\/.
\end{eqnarray}
Thus, as we have mentioned in Introduction, dividing the estimates in \eqref{eq:main} by $(1-|y|)$ and taking a limit as $y\to z\in S(0,1)$ we obtain the following
\begin{corollary}
\label{cor:hxtz:estimates}
For every $T>0$ wa have
\begin{eqnarray}
\label{eq:hxtz:estimates}
q_x(t,z)\approx \left(\frac{1-|x|}{t}+\frac{|x-z|^2}{t}\left(1\wedge \frac{(1-|x|)|x-z|^2}{t}\right)\right)\kernel(t,x,z)
\end{eqnarray}
whenever $|x|<1$, $|z|=1$ and $t<T$.
\end{corollary}
This extends estimates from \cite{S}, where the exit time (without its dependence on exit place) density from the ball was discussed.  
\begin{remark}
  Although the statement of Theorem \ref{thm:main} covers the case of $t<T$ for fixed $T>0$, we emphasize that it is enough to show the estimates for $t$ small enough. Indeed, knowing $\kkt\approx h(t,x,y)\kernel(t,x,y)$ for $t<t_0$ for some $t_0>0$ we can easily replace $t_0$ by $2t_0$ and consequently by any other constant $T>0$. To see that notice the estimates $h(t,x,y)k(t,x,y)\approx (1-|x|)(1-|y|)$ holding whenever $t$ is bounded away from $0$ and infinity. Thus, by the Chapmann-Kolmogorov equation, we simply get
	\begin{eqnarray*}
	   \kk(t,x,y) &=&\int_{B(0,1)}\kk(t/2,x,z)\kk(t/2,z,y)dz\stackrel{t_0}{\approx} (1-|x|)(1-|y|)\int_{B(0,1)}(1-|z|)^2dz\\
		&\stackrel{t_0}{\approx}& h(t,x,y)\kernel(t,x,y)\/,
	\end{eqnarray*}
	whenever $t_0\leq t\leq 2t_0$. Therefore, from now on we will focus only on estimates for $t$ sufficiently small.
\end{remark}

\section{Upper bounds}
We begin with a very simple result providing upper bounds of the following form.
\begin{lemma}
   We have
	\begin{eqnarray}
	   \kkt\lesssim \left(1\wedge \frac{1-|x|}{t}\right)\left(1\wedge \frac{1-|y|}{t}\right)\kt
	\end{eqnarray}
	for every $x,y\in B(0,1)$ and $t>0$.
\end{lemma}

\begin{proof}
  Since $B(0,1)\subset H_y$, we can just write $ \kkt\leq \kernel_{H_y}(t,x,y)$ and consequently
	\begin{eqnarray*}
	  \kkt\lesssim \left(1\wedge \frac{\delta_{H_y}(x)\delta_{H_y}(y)}{t}\right)\kt \lesssim \left(1\wedge\frac{1-|y|}{t}\right)\kt\/,
	\end{eqnarray*}
	by using \eqref{eq:H:estimates} together with $\delta_{H_y}(y)=\del{y}=1-|y|$ and a simple estimate $\delta_{H_y}(x)\leq 2$. Thus, using the Chapmann-Kolmogorov equation and the symmetry of $\kkt$ we arrive at
	\begin{eqnarray*}
	   \kk(2t,x,y) &=& \int_{B(0,1)}\kk(t,x,z)\kk(t,z,y)dz\\
		&\lesssim& \int_{B(0,1)}\left(1\wedge\frac{1-|x|}{t}\right)\kernel(t,x,z)\left(1\wedge\frac{1-|y|}{t}\right)\kernel(t,z,y)dz\\
		&\lesssim& \left(1\wedge\frac{1-|x|}{t}\right)\left(1\wedge\frac{1-|y|}{t}\right)\int_{\R^n}\kernel(t,x,z)\kernel(t,z,y)dz\\
		&=& \left(1\wedge\frac{1-|x|}{t}\right)\left(1\wedge\frac{1-|y|}{t}\right)\kernel(2t,x,y)\/.
	\end{eqnarray*}
	This ends the proof.
\end{proof}
Note that these bounds are optimal for small $t$ if additionally one of the space variables are bounded away from the boundary or $x$ and $y$ are bounded away from each other, i.e. we have
\begin{corollary}
\label{cor:boundedaway}
For a fixed $\varepsilon>0$ we have
\begin{eqnarray*}
 {\kkt}\lesssim h(t,x,y)\kt
\end{eqnarray*}
whenever $1-|x|\geq \varepsilon$ or $|x-y|\geq \varepsilon$.
\end{corollary}
\begin{proof}
 Indeed, if $|x-y|\geq \varepsilon$ we just simply have
\begin{eqnarray*}
  \left(1\wedge \frac{1-|x|}{t}\right)\left(1\wedge \frac{1-|y|}{t}\right)\lesssim \left(1\wedge \frac{(1-|x|)|x-y|^2}{t}\right)\left(1\wedge \frac{(1-|y|)|x-y|^2}{t}\right)
\end{eqnarray*}
and the last expression is apparently dominated by $h(t,x,y)$. Similarly, for $1-|x|\geq \varepsilon$ we have
\begin{eqnarray*}
  \left(1\wedge \frac{1-|x|}{t}\right)\left(1\wedge \frac{1-|y|}{t}\right) \leq \left(1\wedge \frac{1-|y|}{t}\right)
	\lesssim \left(1\wedge \frac{(1-|x|)(1-|y|)}{t}\right)\/,
\end{eqnarray*}
which is smaller than $h(t,x,y)$.
\end{proof}

Note that if the angle $\anxy$ is greater than or equal to $\pi/4$, then we are in the case covered by Corollary \ref{cor:boundedaway}, i.e. if $x$ and $y$ are close to each other and both close to the boundary, then $\anxy$ must be small. Thus, it is enough to prove the upper-bounds with additional assumption that the angle $\anxy$ is smaller than $\pi/4$.

\begin{proposition}
  \label{prop:upper:smallangle}
	There exists a constant $T>0$ such that
	\begin{eqnarray*}
	   \kkt\lesssim h(t,x,y)\kt\/,
	\end{eqnarray*}
	for every $x,y\in B(0,1)$ such that $\anxy<\pi/4$ and $t<T$.
\end{proposition}
\begin{proof}
  Without loss of generality we can assume that 
	\begin{eqnarray*}
  x&=&(x_1,x_2,0\ldots,0)\/,\\
	y&=&(y_1,0\ldots,0)\/,\quad y_1\in[0,1)\/,
\end{eqnarray*}
	and $\del{x}\geq \del{y}$. Since simply $B(0,1)\subset H_x\cap H_y$, we have $\kkt\leq \kernel_{H_x\cap H_y}(t,x,y)$. Moreover, it is clear that
	\begin{eqnarray*}
	   \kernel_{H_x\cap H_y}(t,x,y) = \kernel_{H_{(x_1,x_2)}\cap H_{(y_1,0)}}(t,(x_1,x_2),(y_1,0))\frac{1}{(4\pi t)^{n/2-1}}\exp\left(-\frac{1}{4t}\sum_{k=3}^n (x_k-y_k)^2\right)\/.
	\end{eqnarray*}
	It means that it is enough to consider $2$-dimensional case. Thus, from now on, we will assume that $n=2$, $x=(x_1,x_2)$ and $y=(y_1,0)$.
	
	The proof is divided into two parts. The first one relates to the case when $x\in B(y/(2|y|),1/2)$, i.e. $(x_1-1/2)^2+x_2^2<1/4$. Then $\del{x}\approx \delta_{H_y}(x)$. Indeed, since $y=(y_1,0)$, we have $H_y=H=\{x:\R^2: x_1<1\}$ and consequently $\delta_{H_y}(x)=1-x_1$. The inequality $\del{x}\leq \delta_{H}(x)$ is clear but it is also easy to see that
	\begin{eqnarray*}
	   \delta_{H}(x)= 1-x_1 = 1-x_1^2-\left(\frac14 -\left(x_1-\frac12\right)^2\right)<1-x_1^2-x_2^2 \leq 2(1-|x|)=2\del{x}\/.
	\end{eqnarray*}
Thus, using \eqref{eq:H:estimates}, we obtain
\begin{eqnarray*}
   \kernel_{H_x\cap H}(t,x,y) &\leq& \kernel_{H}(t,x,y)\approx \left(1\wedge\frac{\delta_{H}(x)\delta_{H}(y)}{t}\right)\kt\\
	  &\approx& \left(1\wedge\frac{(1-|x|)(1-|y|)}{t}\right)\kt\leq h(t,x,y)\kt\/.
\end{eqnarray*}	

Now we consider the remaining case, i.e. $x\in B(0,1)\cap [B(y/(2|y|),1/2)]^c$. Since $x$ is outside the ball $B(y/(2|y|),1/2)$, the inversion $x\to x/|x|^2$ transforms the set $B(y/(2|y|),1/2)$ into $H$ and 
\begin{eqnarray*}
   |\bar{x}| = {2-|x|}<\frac{1}{|x|}=\left|\frac{x}{|x|^2}\right|\/,
\end{eqnarray*}
we get that $\bar{x}$ remains inside $H$. Moreover, we have
\begin{eqnarray}
\label{eq:4p:first}
   \kernel_{H_x\cap H}(t,x,y)\leq \kernel_{H}(t,x,y)-\kernel_{H}(t,\bar{x},y).
\end{eqnarray}
To see this, let $A$ be any Borel subset of $B(y/(2|y|)$  and since $\tau_{H_x\cap H}=\tau_{H_x}\wedge \tau_H$ we can write
\begin{eqnarray*}
   \int_A k_{H_x\cap H}(t,x,z)dz &=& \ex^x[t<\tau_{H},W(t)\in A]-\ex^x[\tau_{H_x}<t<\tau_{H}, W(t)\in A]\/.
\end{eqnarray*}
Denoting by $\tau_{H_x}^H$ the first exit time from $H_x$ by the killed process $W^{H}$ we can rewrite the last expression using the strong Markov property in the following way
\begin{eqnarray}
\nonumber
 \ex^x[\tau_{H_x}<t<\tau_{H}, W(t)\in A] &=& \ex^x[\tau_{H_x}^H<t, W^H(t)\in A]\\
\label{eq:4p:proof}
&=& \ex^x\left[\tau_{H_x}^H<t,\ex^{W^H(\tau_{H_x}^H)}\left[W^H(t-\tau_{H_x}^H)\in A\right]\right]\/.
\end{eqnarray}
To make the following computation more transparent we write $P(x)=P_{\partial H_x}(x)=\bar{x}$, i.e. $P$ is the reflection with respect to the hyperplane $\partial H_x$. $P(W)$ is again a Brownian motion, clearly $P(x)=\bar{x}$, $P(H_x)=\textrm{int}(H_x^c)$ and $P(z)=z$ for $z\in \partial H_x$. Moreover, due to the continuity of the paths the first exit times from $H_x^c$ and $\textrm{int}(H_x^c)$ are equal a.s. and we will omit ``$\textrm{int}$'' in the notation. Consequently, for a Borel set $B\in \partial H_x^c$ we have $P(B)=B$. Thus, for every Borel set $I\in (0,t)$ we have
\begin{eqnarray*}
  \ex^{\bar{x}}[\tau_{H_x^c}^H\in I; W^H(\tau_{H_x^c})\lefteqn{\in B] =\ex^{\bar{x}}[\tau_{H_x^c}^H\in I,t<\tau_{H}; W(\tau_{H_x^c})\in B] }
	\\&=&  \ex^{P(x)}[\tau_{P(H_x)}^H\in I,t<\tau_{P(P(H))}; P(W)(\tau_{P(H_x)})\in P(B)]\\
	&=& \ex^{x}[\tau_{H_x}^H\in I, t<\tau_{P(H)}; W(\tau_{H_x})\in B]\/.
\end{eqnarray*}
Moreover, since $\anxy<\pi/4$, then (by simple geometry) $\tau_{P(H)}<\tau_{H}$ on $\{\tau_{H_x}<t\}$ and
\begin{eqnarray*}
 \ex^{\bar{x}}[\tau_{H_x^c}^H\in I; W^H(\tau_{H_x^c})\in B]\leq \ex^{x}[ \tau_{H_x}^H\in I; W^H(\tau_{H_x})\in B]\/.
\end{eqnarray*}
Thus, the last expression in \eqref{eq:4p:proof} is bounded from below by
\begin{eqnarray*}
  \ex^{\bar{x}}\left[\tau_{H_x^c}^H<t,\ex^{W^H(\tau_{H_x^c}^H)}\left[W^H(t-\tau_{H_x^c}^H)\in A\right]\right]\/,
\end{eqnarray*} 
which by the strong Markov property is equal to 
\begin{eqnarray*}
  \ex^{\bar{x}}\left[\tau_{H_x^c}<t<\tau_H, W(t)\in A\right] = \ex^{\bar{x}}\left[t<\tau_H, W(t)\in A\right]\/.
\end{eqnarray*}
Note that since $\bar{x}\in H_x^c$ and $A\subset H_x$, we could remove the condition $\tau_{H_x^c}<t$. Combining all together we arrive at \eqref{eq:4p:first} and thus we have
\begin{eqnarray*}
   \kkt\leq \kt-\kernel(t,\bar{x},y)-\kernel(t,x,\bar{y})+\kernel(t,\bar{x},\bar{y})\/.
\end{eqnarray*}
Since $\angle \bar{x}0\bar{y}=\anxy$ we can find that
\begin{eqnarray*}
   |x-\bar{y}|^2 &=& |x-y|^2+4(1-|y|)(1-|x|\cos\anxy)\/,\\ 
   |\bar{x}-\bar{y}|^2&=& |x-y|^2+4(1-\cos\anxy)((1-|x|)+(1-|y|))\/,
\end{eqnarray*}
which directly lead to
\begin{eqnarray*}
  \kernel(t,x,\bar{y}) &=& \exp\left[-\frac{(1-|y|)(1-|x|\cos\anxy))}{t}\right]\kt\/,\\
  \kernel(t,\bar{x},\bar{y}) &=& \exp\left[-\frac{(1-\cos\anxy)((1-|x|)+(1-|y|))}{t}\right]\kt\/.
\end{eqnarray*}
We can also rewrite $\kt-\kernel(t,\bar{x},y)-\kernel(t,x,\bar{y})$ as
\begin{eqnarray*}
   \left(1-\exp\left[\frac{(1-|x|)(1-|y|\cos\anxy)}{t}\right]-\exp\left[\frac{(1-|y|)(1-|x|\cos\anxy)}{t}\right]\right)\kt
\end{eqnarray*}
and consequently we get the upper bounds for $\frac{\kkt}{\kt}$ as a sum of two components
\begin{eqnarray}
\label{eq:firstcomp}
\left(1-\exp\left[\frac{(1-|x|)(1-|y|\cos\anxy)}{t}\right]\right)\left(1-\exp\left[\frac{(1-|y|)(1-|x|\cos\anxy)}{t}\right]\right)
\end{eqnarray}
and
\begin{eqnarray*}
   \exp\left[-\frac{(1-\cos\anxy)((1-|x|)+(1-|y|))}{t}\right]\left(1-\exp\left[-\frac{2\cos\anxy(1-|x|)(1-|y|)}{t}\right]\right)\/.
\end{eqnarray*}
It is clear that the last expression can be bounded by 
\begin{eqnarray*}
   1\wedge \frac{(1-|x|)(1-|y|)}{t}\/.
\end{eqnarray*}
To deal with the first one note that, by simple geometry, $|x|\sin\anxy\leq |x-y|\leq|y|\tan\anxy$. By our assumptions on $\anxy$ and $|x|$ we get $\sin \anxy\approx |x-y|$ in the considered region. Moreover, since $x$ is outside the ball $B(y/(2|y|),1/2)$ we have $|x|^2\geq x_1$, which simply gives us that $\cos\anxy = \frac{x_1y_1}{|x||y|}=\frac{x_1}{|x|}\leq |x|\leq |y|$. Consequently
\begin{eqnarray*}
   1-\cos\anxy = \frac{\sin^2\anxy}{1+\cos\anxy}\approx |x-y|^2
\end{eqnarray*}
and 
\begin{eqnarray*}
   1-|x|\cos\anxy &=& 1-|x|+|x|(1-\cos\anxy)\approx 1-\cos\anxy\approx |x-y|^2\/.
\end{eqnarray*}
In the similar way we obtain $1-|y|\cos\anxy\approx |x-y|^2$. Combining all together we get the desired bounds for \eqref{eq:firstcomp} and the proof is complete.
\end{proof}

Now we can use the upper bounds from Theorem \ref{thm:main} together with the relation \eqref{eq:hx:diff} to prove the upper bounds in Corollary \ref{cor:hxtz:estimates}. However, since this result will be used in the next section to get the lower bounds of the considered heat kernel $\kkt$, we formulate it in a separate corollary.
\begin{corollary}
\label{cor:hxtz:upper}
For every $T>0$ we have
\begin{eqnarray}
\label{eq:hxtz:upper}
q_x(t,z)\lesssim \left(\frac{1-|x|}{t}+\frac{|x-z|^2}{t}\left(1\wedge \frac{(1-|x|)|x-z|^2}{t}\right)\right)\kernel(t,x,z)
\end{eqnarray}
whenever $|x|<1$, $|z|=1$ and $t<T$.
\end{corollary}


\section{Lower bounds}
It is well-known that whenever the space variables $x$ and $y$ are bounded away from the boundary, the heat kernel is comparable with the Gaussian kernel. We start the proof of the lower bounds with more general result which also ensures comparability between $\kkt$ and $\kt$, but here we only assume that $x$ and $y$ are not to close to the boundary in comparison to the time variable $t$.
\begin{proposition}
\label{prop:lower:rectangle}
For every $C_1>0$ there exists $C_2=C_2(C_1,n)$ such that
\begin{eqnarray*}
  \kkt\geq C_2 \kt
\end{eqnarray*}
for every $x,y$ such that $\del{x}\geq C_1\sqrt{t}$, $\del{y}\geq C_1\sqrt{t}$.
\end{proposition}
\begin{proof}
We begin with considering the cuboid of the form
\formula{
  K = (a_1,b_1)\times (a_2,b_2)\times\ldots\times (a_n,b_n)\/.
}
It is obvious that $\kernel_K(t,x,y)$ is a product of the kernels $\kernel_{(a_i,b_i)}(t,x_i,y_i)$ and in particular, if $\delta_K(x)$ and $\delta_K(y)$ are bounded from below by $c_1 \sqrt{t}$, we get (see \eqref{eq:estimates:d1})
\formula{
  \kernel_K(t,x,y)\geq c_2 \kt\/.
}
Moreover, due to the rotational invariance the same statement is true for every cuboid. To finish the proof it is enough to notice that since $\del{x}$ and $\del{y}$ are greater then $C_1\sqrt{t}$, there exists a cuboid $K$ included in the ball such that $x,y\in K$ and $\delta_K(x), \delta_K(y)\geq c_3\sqrt{t}$ for some positive $c_3$ depending on $C_1$ and $n$. Then, we just can write
	\begin{eqnarray*}
	   \kkt\geq \kernel_K(t,x,y)\geq c_2 \kt\/.
	\end{eqnarray*}
	and the proof is complete.
\end{proof}

The crucial step in the proof of the lower bounds are the estimates when $x$ and $y$ are in a small ball tangent to the sphere $S(0,1)$. In fact we narrow our considerations to $\del{y}<1/16$ and $x\in B(\frac{15}{16}\frac{y}{|y|},\frac1{16})$. Note that in this case, we have
\begin{eqnarray*}
|x|^2<1-|x-\frac{y}{|y|}|^2<1-|x-y|^2\/,
\end{eqnarray*}
where the last inequality holds if $\del{x}\geq \del{y}$. Consequently
\begin{eqnarray*}
   \del{x}=1-|x|> \frac{|x-y|^2}{1+|x|}> \frac{1}{2}|x-y|^2\/. 
\end{eqnarray*}
and obviously the first component on the right-hand side of \eqref{eq:main} dominates the other. Thus, our next aim is to proof the following
\begin{proposition}
\label{prop:smallball}
There exist constants $C_3=C_3(n)>0$, $t_0=t_0(n)>0$ and $m=m(n)>0$ such that 
	\begin{eqnarray*}
	   \kkt\geq C_3\, \left(1\wedge \frac{(1-|x|)(1-|y|)}{t}\right)\kt\/,
	\end{eqnarray*}
for every $x,y\in B(0,1)$ such that $\del{y}<1/16$, $x\in B(\frac{15}{16}\frac{y}{|y|},\frac{1}{16})$, $\del{x}\geq m\sqrt{t}$ and $t<t_0$.
\end{proposition}
\begin{proof}
The best way to present the technical details of the proof and to make it more transparent and simpler to read is to consider the ball $B=B((1,0,\ldots,0),1)$ and set $y=(y_1,0,\ldots,0)$, where in general $y_1\in(0,1/16)$. Note that such a choice implies that $H_y=\{x\in\R^n: x_1>0\}$ and for simplicity we denote it by $\Hh$. Moreover, we set $x=(x_1,x_2,0,\ldots,0)$ and assume as previously that $\delB{x}\geq \delB{y}$. Our assumptions now reads as $x\in B((1/16,0,\ldots,0),1/16)$ and it implies that $x_1^2+x_2^2\leq x_1/8$. Thus 
\begin{eqnarray}
\label{eq:delB:x1}
   \delB{x} \approx 1-[(x_1-1)^2+x_2^2] = 2x_1-(x_1^2+x_2^2)\approx x_1\/.
\end{eqnarray}
Consequently, since $\delta_{\Hh}(x)=x_1$ and $\delta_{\Hh}(y)=y_1$ we have
\begin{eqnarray}
   \label{eq:p_H:est_recall}
   \kernel_{\Hh}(t,x,y)\approx  \left(1\wedge \frac{x_1y_1}{t}\right)\kt \approx \left(1\wedge \frac{\delB{x}\delB{y}}{t}\right)\kt\/.
\end{eqnarray}
Moreover, we have
\begin{eqnarray*}
   \ex^x(\tau_B<t;W(t)\in dy)=\ex^x(\tau_{\Hh}<t;W(t)\in dy)-\ex^x(\tau_{B}<t<\tau_{\Hh}; W(t)\in dy)\/,
\end{eqnarray*}
thus it is enough to show that 
\begin{eqnarray*}
   R(t,x,y) := \ex^x(\tau_{B}<t<\tau_{\Hh}; W(t)\in dy)/dy
\end{eqnarray*}
is dominated by $c\,\kernel_{\Hh}(t,x,y)$ for some $c<1$. By Strong Markov property we can write
	\begin{eqnarray*}
	   R(t,x,y)dy = \ex^x[\tau_B<t;\ex^{W(\tau_B)}(t<\tau_{\Hh};W(t-\tau_B)\in dy)]
	\end{eqnarray*}
	and consequently
	\begin{eqnarray*}
	   R(t,x,y) &=& \int_0^t \int_{\partial B} q_x(s,z)\kernel_{\Hh}(t-s,z,y)\,dsd\sigma(z)\\
		   &=& \int_0^t \left(\int_{A_1(s)}+\int_{A_2(s)}\right)q_x(s,z)\kernel_{\Hh}(t-s,z,y)\,dsd\sigma(z)\\
			&=:& R_1(t,x,y)+R_2(t,x,y)\/,
	\end{eqnarray*}
	where $A_1(s)=\{z\in \partial B: |x-z|^4\geq 25(\deldB{x}\vee s)\}$ and $A_2(s)=\partial B\setminus A_1(s)$. Note that the second term of the right-hand side of (\ref{eq:htxy:est}) dominates on the set $A_1(s)$, i.e.   
	\begin{eqnarray*}
	   q_x(s,z)\lesssim \frac{|x-z|^2}{s}\left(1\wedge \frac{\delB{x}|x-z|^2}{s}\right)\kernel(s,x,z)\/,\quad z\in A_1(s)\/.
	\end{eqnarray*}
	Moreover, since $|x-z|^2\geq 5\delB{x}\geq \frac52 |x-y|^2$ on $A_1(s)$, we have $|x-z|^2\geq \frac{5}{4}|x-y|^2+\frac{1}{2}|x-z|^2$ and $|x-z|^2\leq 2|x-y|^2+2|y-z|^2\leq \frac45|x-z|^2+2|y-z|^2$, which implies $|x-z|^2\leq 10|y-z|^2$. Thus, using the above-given estimates together with \eqref{eq:p_H:est_recall}, then replacing the set $A_1(s)$ simply by $\partial B$ and interchanging the integrals,  we arrive at
	\begin{equation*}
	   R_1(t,x,y)\lesssim \delB{x} \delB{y}e^{-\frac{5}{4}\frac{|x-y|^2}{4t}}\int_{\partial B}|x-z|^4z_1\int_0^t \frac{1}{s^{n/2+2}(t-s)^{n/2+1}}e^{-\frac{|x-z|^2}{8s}}e^{-\frac{|x-z|^2}{40(t-s)}}dsd\sigma(z)\/.
	\end{equation*}
	We have used here the fact that $\delta_{\Hh}(z)=z_1$ and $\delta_{\Hh}(y)=\delB{y}$. Moreover, we have
	\begin{eqnarray*}
	   \int_0^{t/2}\frac{|x-z|^4}{s^{n/2+2}(t-s)^{n/2+1}}\lefteqn{e^{-\frac{|x-z|^2}{8s}}e^{-\frac{|x-z|^2}{40(t-s)}}ds\lesssim \frac{|x-z|^4}{t^{n/2+1}}\int_0^{t/2}\frac{1}{s^{n/2+2}}e^{-\frac{|x-z|^2}{8s}}ds}\\
	&\leq&\frac{|x-z|^{4-n}}{t^{n/2+1}} \sup_{r>0}r^{n/2}e^{-r/16}\,\int_0^{t/2}\frac{1}{s^{2}}e^{-\frac{|x-z|^2}{16s}}ds	\lesssim \frac{|x-z|^{2-n}}{t^{n/2+1}}\/.
	\end{eqnarray*}
	Similarly
	\begin{eqnarray*}
	   \int_{t/2}^t\frac{|x-z|^4}{s^{n/2+2}(t-s)^{n/2+1}}\lefteqn{e^{-\frac{|x-z|^2}{8s}}e^{-\frac{|x-z|^2}{40(t-s)}}ds \lesssim \frac{|x-z|^4}{t^{n/2+2}}\int_0^{t/2}\frac1{u^{n/2+1}}e^{-\frac{|x-z|^2}{40u}}du}\\
	   &\lesssim&	 \frac{|x-z|^{6-n}}{t^{n/2+2}}\int_0^{t/2}\frac1{u^2}e^{-\frac{|x-z|^2}{40u}}du	
	   \approx\frac{|x-z|^{4-n}}{t^{n/2+2}}e^{-\frac{|x-z|^2}{20t}}\lesssim\frac{|x-z|^{2-n}}{t^{n/2+1}}\/.
	\end{eqnarray*}
	Furthermore, since $\sup_{x\in B}\int_{\partial B}|x-z|^{2-n}dz<\infty$,  we obtain
	\begin{eqnarray*}
	   R_1(t,x,y)\lesssim \frac{\delB{x}\delB{y}}{t}\kt e^{-\frac{|x-y|^2}{16t}}\/.
	\end{eqnarray*}
	In the same way, using
	\begin{eqnarray*}
	   q_x(s,z)\lesssim \frac{|x-z|^2}{s}\kernel(s,x,z)
	\end{eqnarray*}
	and $\kernel_{\Hh}(t-s,z,y)\leq \kernel(t-s,z,y)$ we arrive at
	\begin{eqnarray*}
	   R_1(t,x,y)\lesssim e^{-\frac{5}{4}\frac{|x-y|^2}{4t}}\int_{A_1}|x-z|^2z_2\int_0^t \frac{1}{s^{n/2+1}(t-s)^{n/2}}e^{-\frac{|x-z|^2}{8s}}e^{-\frac{|x-z|^2}{40(t-s)}}dsd\sigma(z)
	\end{eqnarray*}
	and we can similarly show that
		\begin{eqnarray*}
	   R_1(t,x,y)\lesssim \kt e^{-\frac{|x-y|^2}{16t}}\/.
	\end{eqnarray*}
	It finally gives
			\begin{eqnarray*}
	   R_1(t,x,y)\lesssim \left(1\wedge\frac{\delB{x}\delB{y}}{t}\right)\kt e^{-\frac{|x-y|^2}{16t}}<\frac{1}{3}\kernel_{\Hh}(t,x,y)
	\end{eqnarray*}
	for $|x-y|^2/t$ large enough. Note that if $|x-y|^2/t$ is bounded, then there is no exponential decay of the kernel $\kkt$ and consequently the lower bounds we want to show are just given in the Zhang's result \eqref{eq:Zhang:est}.
	
	The estimates of $R_2(t,x,y)$ are much more delicate. We simply begin with
	\begin{eqnarray*}
	  \kernel_{\Hh}(t,z,y) \lesssim  \frac{z_1\delB{y}}{t}\kernel(t,z,y)\/.
	\end{eqnarray*}
	Note that for $z\in A_2(s)$ we have $q_x(s,y)\lesssim \frac{\delB{x}}{s}\kernel(s,x,z)$ and consequently
	\begin{eqnarray*}
	  R_2(t,\lefteqn{x,y) = \int_0^t \int_{A_2(s)} q_x(s,z)\kernel_{\Hh}(t-s,z,y)\,dsd\sigma(z)}\\
		&\lesssim &\delB{x}\delB{y}\int_0^t \int_{A_2(s)}\frac{z_1}{s^{n/2+1}(t-s)^{n/2+1}}\exp\left(-\frac{|x-z|^2}{4s}-\frac{|z-y|^2}{4(t-s)}\right)d\sigma(z)\/.
	\end{eqnarray*}
	Since we have $\sqrt{5\delB{x}}\leq\sqrt{5/8}<0.8$ and $\sqrt[4]{25s}\leq \sqrt[4]{25t}<\sqrt[4]{25/62}<0.8$ for $t<1/62$, we get that $|x-z|<0.8$ whenever $z\in A_2(s)$. Thus $|z|\leq |x-z|+|x|<0.8+1/8=0.925$. In particular, it implies that $z_2^2+\ldots+z_n^2<c<1$, where $c=(0.925)^2$. It means that if we parametrize this part of the sphere writing $z_1=f(z_2,\ldots,z_n) = 1-\sqrt{1-(z_2^2+\ldots+z_n^2)}$ we get that 
	\begin{eqnarray*}
	   \sqrt{1+\left(\frac{\partial f}{\partial z_2}\right)^2+\ldots+\left(\frac{\partial f}{\partial z_n}\right)^2}
	\end{eqnarray*}
	is comparable with a constant and consequently we can write 
\begin{eqnarray*}
 R_2(t,x,y)\lesssim \delB{x}\delB{y}\int_0^t \int_{|\tilde{z}|\leq c}\frac{z_1}{s^{n/2+1}(t-s)^{n/2+1}}\exp\left(-\frac{|x-z|^2}{4s}-\frac{|z-y|^2}{4(t-s)}\right)d\tilde{z}\/,
\end{eqnarray*}
where $\tilde{z}=(z_2,\ldots,z_n)$ is the projection from $\R^n$ to $\R^{n-1}$. First we simply estimate $z_1=1-\sqrt{1-|\tilde{z}|^2}\leq |\tilde{z}|^2$. Next we can write
\begin{eqnarray*}
   \frac{|x-z|^2}{4s}+\frac{|z-y|^2}{4(t-s)} &=&\frac{(x_1-z_1)^2+(x_2-z_2)^2+z_3^2+\ldots+z_n^2}{4s}+\frac{(y_1-z_1)^2+|\tilde{z}|^2}{4(t-s)}\\
	&\geq&\frac{x^2_2-2z_2x_2+|\tilde{z}|^2+x^2_1-2x_1z_1}{4s}+\frac{|\tilde{z}|^2}{4(t-s)}
	\end{eqnarray*}
	by omitting two non-negative terms. Since $z_1\leq |\tilde{z}|^2$ we get
	\begin{eqnarray*}
	   \frac{|\tilde{z}|^2-2x_1z_1}{4s}+\frac{|\tilde{z}|^2}{4(t-s)}\geq \frac{|\tilde{z}|^2}{4s}\left(\frac{t}{t-s}-2x_1\right)
	\end{eqnarray*}
	Now we put $w:=\frac{t}{t-s}-2x_1$. Notice that since $t>s$ and $x_1<1/8$ we have $w\geq \frac{t}{2(t-s)}\geq \frac12$ and in particular $w$ is strictly positive. Thus we can write
	\begin{eqnarray*}
   \frac{|x-z|^2}{4s}+\frac{|z-y|^2}{4(t-s)} &\geq&
	\frac1{4s}\left(w|\tilde{z}|^2-2z_2x_2+x_2^2+x^2_1\right)\\
	&=& \frac{1}{4s}\left(w\left|\tilde{z}-\frac{1}{w}\tilde{x}\right|^2+x_2^2\left(1-\frac{1}{{w}}\right)+x_1^2\right)\/,
\end{eqnarray*}
where we have used the special form of $x$ to notice that $z_2x_2$ is just the inner product of $\tilde{z}$ and $\tilde{x}$. Moreover, using $w\geq 1/2$, we get
\begin{eqnarray*}
   1-\frac{1}{w} = \frac{s}{t}-\frac{t-s}{t}\frac{2x_1}{w}\geq \frac{s}{t}-\frac{4x_1(t-s)}{t}\/.
\end{eqnarray*}
Since $x\in B((0,1/16),1/16)$, we have $x_2^2\leq x_1^2+x_2^2\leq x_1/8$, which implies
\begin{eqnarray*}
  \frac1{4s}\left(x_2^2\left(1-\frac{1}{{w}}\right)+x_1^2\right)\geq\frac{x_1^2+x_2^2}{4t}+\frac{x_1^2(t-s)}{4st}-\frac{4x_1x_2^2(t-s)}{4st}\geq \frac{x_1^2+x_2^2}{4t} + \frac{x_1^2(t-s)}{8st}\/.
\end{eqnarray*}
Finally, using $w\geq \frac{t}{2(t-s)}$, we obtain
\begin{eqnarray*}
\int_{|\tilde{z}|\leq c}|\tilde{z}|^2\exp\left(-\frac{w}{4s}\left|\tilde{z}-\frac{1}{w}\tilde{x}\right|^2\right)d\tilde{z}&\lesssim& \left(\frac{s}{w}\right)^{(n-1)/2}\left({\frac{s}{w}}+\left|\frac{\tilde{x}}{w}\right|^2\right)\\
&\lesssim& \left(\frac{s(t-s)}{t}\right)^{(n-1)/2}\left(\frac{s(t-s)}{t}+\frac{x_2^2(t-s)^2}{t^2}\right)\/,
\end{eqnarray*}
where we just enlarge the region of the integration to the whole $\R^{n-1}$. Combining all together we arrive at
\begin{eqnarray*}
  R_2(t,x,y) &\lesssim& \frac{\delB{x}\delB{y}}{t^{(n+3)/2}}\exp\left({-\frac{|x|^2}{4t}}\right)\int_0^t \frac{st+x_2^2(t-s)}{s^{3/2}(t-s)^{1/2}}\exp\left(-\frac{x_1^2(t-s)}{8st}\right)ds\/.
\end{eqnarray*}
Then, substituting $u=x_1^2(1/s-1/t)$ we reduce the the right-hand side of the above-given inequality to
\begin{eqnarray*}
 \frac{\delB{x}\delB{y}}{t^{(n+2)/2}}\exp\left({-\frac{|x|^2}{4t}}\right)\,x_1\,\int_0^\infty \frac{1+(x_2/x_1)^2u}{u^{1/2}(u+x_1^2/t)}\exp\left(-\frac{u}{8}\right)du
\end{eqnarray*}
and the last integral can be easily bounded from above by 
\begin{eqnarray*}
   \frac{t}{x_1^2}\left(1+\left(\frac{x_2}{x_1}\right)^2\right)\int_0^\infty \frac{1+u}{\sqrt{u}}\exp\left(-\frac{u}{8}\right)du\/.
\end{eqnarray*}
Since $\delB{x}=x_1$ (see \eqref{eq:delB:x1}), $\delB{y}=y_1$, $x_1^2+x_2^2\leq x_1/8$ and $|x|^2=|x-y|^2+y_1(2x_1-y_1)\geq |x-y|^2+x_1y_1$ we obtain
\begin{eqnarray*}
  R_2(t,x,y)\lesssim \frac{x_1y_1}{t}\, \frac{t}{x_1^2} \exp\left(-\frac{x_1y_1}{4t}\right)\kt
\end{eqnarray*}
and it is clear that for every $c>0$ we can chose $m>0$ such that 
\begin{eqnarray*}
  R_2(t,x,y)\leq \frac{c}{2}\left(1\wedge\frac{x_1y_1}{t}\right)\kt
\end{eqnarray*}
for every $x$ such that $\delB{x}\geq m\sqrt{t}$. This ends the proof.

\end{proof}

The next step is to show that we can enlarge the ball $B(\frac{15}{16}\frac{y}{|y|},\frac{1}{16})$ considered in Proposition \ref{prop:smallball} to the ball $B(y/(3|y|),2/3)$ i.e. we prove the following
\begin{proposition}
\label{prop:ball12}
   There exist $C_4=C_4(n)>0$, $m=m(n)>0$ and $t_0=t_0(n)>0$ such that 
	\begin{eqnarray}
			\label{eq:lowerbounds:01}
	   \kkt\geq C_4 \left(1\wedge \frac{(1-|x|)(1-|y|)}{t}\right)\kt\/,
	\end{eqnarray}
	whenever $x\in B(y/(3|y|),2/3)$, $\del{x}\geq m\sqrt{t}$ and $t<t_0$.
\end{proposition}
\begin{proof}
    The idea of the proof is to show that if we know the lower bounds of the form \eqref{eq:lowerbounds:01} for every $x, y \in B((1-r)y/|y|,r)$ for some $r>0$, such that $\del{x}\geq M\sqrt{t}$ for some $M>0$ and $\del{x}\geq \del{y}$, then we can deduce the estimates of the same form for the same range of parameters but with $r$ replaced by $3r/2$ and possibly with different constants $M$ and $t_0$. Applying this procedure $6$ times ($(3/2)^6/16>2/3$) we will then get Proposition \ref{prop:ball12} from Proposition \ref{prop:smallball}. 
		
		Consequently, our starting point are the estimates \eqref{eq:lowerbounds:01} holding for $x$ and $y$ as stated above for some $r\in(0,2/3)$ with some $M>0$ and $t_0>0$. Let $x,y\in B((1-R)y/|y|,R)$, where $R=3r/2$, such that $\del{x}\geq \del{y}$ and $\del{x}\geq m\sqrt{t}$, where $m=8M$. We can additionally assume that $y$ is close to boundary by requiring $\del{y}<\sqrt{t}<R/3$, since the case $\del{x}\geq \del{y}\geq \sqrt{t}$ follows directly from Proposition \ref{prop:lower:rectangle}. Note that under our assumption, the midpoint between $x$ and $y$ belongs to the ball $B((1-r)y/|y|,r)$. Indeed, since $|y|> 1-R/3=1-r/2$, we have
		\begin{eqnarray*}
		    \frac{1}{2}-\frac{1-R}{2|y|}-\frac{R-r}{|y|}= \frac{|y|-1+r/2}{2|y|}>0
		\end{eqnarray*}
		and consequently
	\begin{eqnarray*}
     \left|\frac{x+y}{2}-\frac{(1-r)y}{|y|}\right| &=&  	 \left|\frac{x}{2}-\frac{(1-R)y}{2|y|}+y\left(\frac{1}{2}-\frac{(1-R)}{2|y|}-\frac{R-r}{|y|}\right)\right| \\
		&\leq &\frac{1}{2}\left|x-\frac{(1-R)y}{|y|}\right|+|y|\left(\frac{1}{2}-\frac{1-R}{2|y|}-\frac{R-r}{|y|}\right)\\
		&=& \frac{1}{2}\left|x-\frac{(1-R)y}{|y|}\right|+\frac{1}{2}\left|y-\frac{(1-R)y}{2|y|}\right|-(R-r)\\
		&\leq& \frac{R}{2}+\frac{R}{2}-(R-r) = r\/.
		\end{eqnarray*}
		We fix $t_1=t_0\wedge (8r/m)^2$ and for $t<t_1$ we consider $A$ as a point on the line between $(x+y)/2$ and $(1-r)y/|y|$ such that $(x+y)/2$ belongs to the sphere $S(A,m\sqrt{t}/16)$. Such a choice ensures that the ball $B(A,m\sqrt{t}/16)$ is contained in $B((1-r)y/|y|,r)$ as well. Moreover, since we have $\del{(x+y)/2}\geq \del{x}/4$ (see \eqref{eq:parallel:low}), for every $z\in B(A,m\sqrt{t}/4)$ we have 
		\begin{eqnarray*}
		   \del{z}\geq \del{\frac{x+y}{2}}-\frac{m\sqrt{t}}{8}\geq \frac{\del{x}}{8}\geq\frac{m\sqrt{t}}{8}=M\sqrt{t}\/.
		\end{eqnarray*}
		We use the Chapmann-Kolmogorov equation to get
		\begin{eqnarray*}
   \kk(t,x,y)&\geq& \int_{B(A,m\sqrt{t}/16)}\kk(t/2,x,z)\kk(t/2,z,y)dz\/.
	\end{eqnarray*}
		We have $\kk(t,x,z)\gtrsim \kernel(t,x,z)$ since $\del{z}\geq M\sqrt{t}$ and $\del{x}\geq m\sqrt{t}$. Moreover, since $y,z\in B((1-r)y/|y|,r)$, we have also
		\begin{eqnarray*}
		   \frac{\kk(t,z,y)}{\kernel(t,z,y)}\gtrsim \left(1\wedge \frac{\del{z}\del{y}}{t}\right)\gtrsim  \left(1\wedge \frac{\del{x}\del{y}}{t}\right)\/.
		\end{eqnarray*}
		Using these estimates and Lemma \ref{lemma:kt:lower} (note that $|A-(x+y)/2|=m\sqrt{t}/16$) we obtain the desired lower bounds for $x$ and $y$ in the larger ball. 
\end{proof}

To make the last step of the proof we consider two points $x,y\in B(0,1)$ and two balls $B(x/(3|x|),2/3)$ and $B(y/(3|y|),2/3)$. It is geometrically clear that the midpoint $(x+y)/2$ as well as $0$ belong to both of them. In view of Proposition \ref{prop:lower:rectangle}, we can additionally assume that one of the variables is close to the boundary, i.e. $\del{y}\leq \sqrt{t}$. Now we consider two cases. The first one relates to the situation when the midpoint is close to the origin, i.e. $(x+y)/2\in B(0,1/6)$. Then 
\begin{eqnarray*}
  B((x+y)/2,1/6)\subset B(0,1/3)\subset B(x/(3|x|),2/3) \cap B(y/(3|y|),2/3)
	\end{eqnarray*}
The Chapman-Kolmogorov once again implies that
\begin{eqnarray*}
   \kk(2t,x,y)\geq \int_{B((x+y)/2,1/6)}\kk(t,x,z)\kk(t,z,y)dz 
\end{eqnarray*}
We can use Proposition \ref{prop:ball12} to write $\kk(t,x,z)\gtrsim (1\wedge \del{x}/t)\kernel(t,x,z)$ and $\kk(t,z,y)\gtrsim (1\wedge \del{y}/t)\kernel(t,z,y)$, since $\del{z}>5/6>m\sqrt{t}$, where $m$ is the constant  from Proposition \ref{prop:ball12} and $t$ is small enough. Thus we obtain
\begin{eqnarray*}
   \kk(2t,x,y) &\gtrsim&  \left(1\wedge \frac{\del{x}}{t}\right)\left(1\wedge \frac{\del{y}}{t}\right)\int_{B((x+y)/2,1/6)}\kernel(t,x,z)\kernel(t,z,y)dz \\
	&\gtrsim& \left(1\wedge \frac{\del{x}}{t}\right)\left(1\wedge \frac{\del{y}}{t}\right)\kernel(t,x,y)
\end{eqnarray*}
where we used Lemma \ref{lemma:kt:lower}. Note that if $y$ is close to the boundary and $(x+y)/2$ is close to the origin, then $|x-y|$ is bounded away from $0$, which implies $h(t,x,y)\approx(1\wedge \del{x}/t)(1\wedge \del{y}/t)$. It ends the proof in this case. 

In the remaining case, i.e. when $|(x+y)/2|\geq 1/6$ we set $t_1=1/(18m)^2 \wedge t_0$ and consider $t<t_1$. Here $m$ and $t_0$ are the constant from Proposition \ref{prop:ball12}. We also define $A$ to be a point on the line between $(x+y)/2$ and $0$ such that $|A-(x+y)/2|=2m\sqrt{t}$. Then we have $B(A,m\sqrt{t})\subset B(x/(3|x|),2/3) \cap B(y/(3|y|),2/3)$ and for every $z\in B(A,m\sqrt{t})$ we get $\del{z}\geq \del{(x+y)/2}+m\sqrt{t}$. Thus we write for the last time that
\begin{eqnarray*}
\kk(2t,x,y) &\geq& \int_{B(A,\sqrt{t})}\kk(t,x,z)\kk(t,z,y)dz \\
&\geq& \int_{B(A,\sqrt{t})}\left(1\wedge \frac{\del{x}\del{z}}{t}\right)\left(1\wedge \frac{\del{y}\del{z}}{t}\right)\kernel(t,x,z)\kernel(t,z,y)dz
\end{eqnarray*}
If we use the estimates $\del{z}\gtrsim \del{(x+y)/2}\gtrsim |x-y|^2$ (by \eqref{eq:parallel:low}) and Lemma \ref{lemma:kt:lower}, we arrive at
\begin{eqnarray*}
  \kk(2t,x,y) \gtrsim \left(1\wedge \frac{\del{x}|x-y|^2}{t}\right)\left(1\wedge \frac{\del{x}|x-y|^2}{t}\right)\kernel(2t,x,y)\/.
\end{eqnarray*}
On the other hand, if $\del{x}\leq \sqrt{t}$ we can write
\begin{eqnarray*}
    \left(1\wedge \frac{\del{x}\del{z}}{t}\right)\left(1\wedge \frac{\del{y}\del{z}}{t}\right)\gtrsim \frac{\del{x}}{\sqrt{t}}\frac{\del{y}}{\sqrt{t}} \approx \left(1\wedge \frac{\del{x}\del{y}}{t}\right)\/,
\end{eqnarray*}
which gives $k(2t,x,y)\gtrsim h(2t,x,y)k(2t,x,y)$. Finally, for $\del{x}\geq \sqrt{t}\geq \del{y}$ we have $\del{x}\del{z}\geq mt$ and consequently 
\begin{eqnarray*}
    \left(1\wedge \frac{\del{x}\del{z}}{t}\right)\left(1\wedge \frac{\del{y}\del{z}}{t}\right)\gtrsim \left(1\wedge \frac{\del{y}\del{z}}{{t}}\right) \gtrsim \left(1\wedge \frac{\del{x}\del{y}}{t}\right)\/,
\end{eqnarray*}
since $\del{z}\gtrsim \del{(x+y)/2}\gtrsim \del{x}$ by \eqref{eq:parallel:low}. This ends the proof.

 \bibliography{bibliography}    
\bibliographystyle{plain}

\end{document}